\documentclass{amsart}
\usepackage{graphicx}
\usepackage{tikz}
\usepackage{hyperref}
\usepackage{comment}
\usepackage[T1]{fontenc}

\newtheorem{thm}{Theorem}[section]
\newtheorem{prop}[thm]{Proposition}
\newtheorem{lem}[thm]{Lemma}
\newtheorem{cor}[thm]{Corollary}

\theoremstyle{definition}
\newtheorem{definition}[thm]{Definition}
\newtheorem{example}[thm]{Example}

\theoremstyle{remark}

\newtheorem{remark}[thm]{Remark}

\numberwithin{equation}{section}

\newcommand{\R}{\mathbb{R}}
\newcommand{\T}{\mathcal{T}}

\DeclareMathOperator{\alt}{alt}
\DeclareMathOperator{\supp}{supp}
\DeclareMathOperator{\vol}{vol}
\DeclareMathOperator{\id}{id}
\DeclareMathOperator{\Ker}{Ker}
\DeclareMathOperator{\Image}{Im}
\DeclareMathOperator{\Homeo}{Homeo}
\DeclareMathOperator{\Diff}{Diff}
\DeclareMathOperator{\Symp}{Symp}
\DeclareMathOperator{\sgn}{sgn}

\DeclareMathOperator{\Flux}{Flux}

\subjclass[2000]{57S05, 20J06}
\keywords{bounded cohomology; relative quasimorphisms; diffeomorphism groups}

\begin{document}

\title{Norm-controlled cohomology of transformation groups}

\author{Mitsuaki Kimura}
\address{Department of Mathematics, Kyoto University, Kitashirakawa Oiwake-cho, Sakyo-ku, Kyoto 606-8502, Japan}
\email{mkimura@math.kyoto-u.ac.jp}

\begin{abstract}
We generalize the result of Brandenbursky and Marcinkowski for the bounded cohomology of transformation groups to the infinite volume case.
To state the result, we introduce the notion of norm-controlled cohomology as a generalization of bounded cohomology.
This cohomology theory also provides a framework for relative quasimorphisms.
\end{abstract}

\maketitle

\section{Introduction} \label{section:intro}

Since Gromov's seminal paper \cite{MR686042}, bounded cohomology has received considerable attention and has been studied by many authors. However, its computation is difficult in general.
In particular, bounded cohomology of third degree or higher had rarely been calculated.
%For the real coefficient case, the first cohomology is trivial.
%The second cohomology can be highly non-trivial (typically infinite-dimensional).
%There are many studies on the computation of the second bounded cohomology.
%On the other hand, the third bounded cohomology had rarely been computed, except for Soma's result \cite{MR1455524}. In 2015, Frigerio, Pozzetti, and Sisto \cite{MR3431671} proved that the third bounded cohomology is infinite-dimensional for acylindrically hyperbolic groups.
Recently, Brandenbursky and Marcinkowski \cite{brandenbursky2019bounded} proved that the third bounded cohomology is infinite-dimensional for certain transformation groups of non-positively curved manifolds.
In this paper, we generalize the result of Brandenbursky and Marcinkowski by introducing a new cohomology theory.

Let $M$ be a complete connected Riemannian manifold and $\mu$ the measure on $M$ induced by the Riemannian metric.
Let $\Homeo_0(M,\mu)$ denote the identity component of the group of measure-preserving homeomorphisms of $M$ with compact support. Here, we endow $\Homeo_0(M,\mu)$ with the compact-open topology.
In \cite{brandenbursky2019bounded}, Brandenbursky and Marcinkowski gave a construction of bounded cohomology classes
 of $\Homeo_0(M,\mu)$ for the case when $M$ is of finite volume. We consider the following problem: what happens if $M $ is of infinite volume?
We observe that their construction gives not necessarily a bounded cochain but a ``norm-controlled'' cochain (Proposition \ref{prop:main_observation}).
In order to regard this cochain as a certain cohomology class, we introduce a notion of norm-controlled cohomology. 
Here we review the definition of a norm on a group.
Let $G$ be a group and $1_G$ denote the identity element of $G$. A function $\nu \colon G \to [0,\infty)$ is a \emph{norm} if it satisfies
\begin{enumerate}
  \item $\nu(gh)\leq \nu(g)+\nu(h)$ for any $g,h \in G$,
  \item $\nu(g^{-1})=\nu(g)$ for any $g \in G$,
  \item $\nu(1_G)=0$,
  \item $\nu(g)>0$ if $g \neq 1_G$.
\end{enumerate}
If one drops condition (4), $\nu$ is said to be a \emph{pseudo norm}.
%We refer to a group equipped with a pseudo norm as a \emph{normed group}.
If a norm $\nu$ additionally satisfies the condition $\nu(ghg^{-1})=\nu(h)$ for every $g,h \in G$, $\nu$ is called a \emph{conjugation-invariant norm}.
The concept of conjugation-invariant norm was introduced and studied by Burago, Ivanov, and Polterovich in \cite{MR2509711}.
The definition (of the simplest version) of norm-controlled cohomology is the following.

\begin{definition}
Let $G$ be a group and $\nu$ a (pseudo) norm on $G$. %(see \S \ref{subsection:def_of_coh}).
We define $\bar{C}_{\nu}^n(G)$ to be the set of (inhomogeneous) cochains $\bar{c} \in \bar{C}^n(G)$ satisfying the following:
there exist  $C,D \geq 0$ such that for all  $g_1 ,\dots ,g_n \in G$,
\[ |\bar{c}(g_1,\dots,g_n)| \leq C(\nu(g_1)+\dots + \nu(g_n)) + D. \]
\emph{The norm-controlled cohomology}, denoted by $H_{\nu}^n(G)$, is defined to be the cohomology of the cochain complex $(\bar{C}_{\nu}^n(G), \bar{\delta})$.
\emph{The exact norm-controlled cohomology} $EH_{\nu}^n(G)$ is the kernel of the comparison map $H_{\nu}^n(G) \to H^n(G)$.
\end{definition}

Here we only consider the cohomology with real coefficients. See Section \ref{section:group_coh} for more details on group cohomology. If $\nu$ is a bounded norm, then $H_{\nu}^n(G)$ is nothing but the bounded cohomology $H_b^n(G)$ of $G$.
Note that a similar generalization of bounded cohomology is studied for finitely generated groups and its word length, which is called the \emph{polynomially bounded cohomology} (see \cite{MR2109110} for example).

We mainly consider several subgroups of $\Homeo_0(M,\mu)$.
Let $\Diff_0(M, \vol)$ denote the identity component of the group of diffeomorphisms with compact support which preserve the volume form induced by the Riemannian metric. If $(M,\omega)$ is a symplectic manifold, let $\Symp_0(M,\omega)$ denote the identity component of the group of symplectomorphisms with compact support.
We regard $\Symp_0(M,\omega)$ as a subgroup of $\Homeo_0(M,\mu)$ by considering $\mu$ induced by the volume form $\omega^{(\dim M)/2}$.

In this paper, we consider norm-controlled cohomology with respect to the fragmentation norm defined as follows.

\begin{definition} \label{def:frag}
Let $U$ be an open subset of $M$. Let $\mathcal{S}_U$ denote the set of elements $h \in \T_M$ that satisfy the following condition: there exists an isotopy $\{h_t\}_{0\leq t\leq 1}$ of $h$ such that $\supp(h_t) \subset U$ for every $t\in[0,1]$.
We define the \emph{fragmentation norm $\nu_U$ with respect to $U$}  on $\T_M$ by
\begin{equation*}
  \nu_U(g) = \min
  \left\{ k \; \middle|
  \begin{gathered}
    \begin{gathered}\end{gathered}
    \exists f_i \in \T_M, \exists h_i \in \mathcal{S}_U, (i=1,\dots,k) \\
    g=(f_1^{-1} h_1 f_1) \cdots (f_k^{-1} h_k f_k)
  \end{gathered}
  \right\}
\end{equation*}
for $g \in \T_M$. If no such decomposition of $g$ exists, we define $\nu_U(g)=+\infty$.
We say that $\nu_U$ is \emph{well-defined on $\T_M$} if $\nu_U(g)<+\infty$ for all $g\in \T_M$.
\end{definition}

\begin{example}
Let $M$ be a manifold, $U$ a non-empty open subset of $M$, and $i \colon U \to M$ the inclusion.

\begin{itemize}

\item Let $\T_M$ be $\Symp_0(M,\omega)$ and $\Flux_{\omega} \colon \T_M \to H^1_c(M;\R)/\Gamma_{\omega}$ denote the symplectic flux homomorphism, where $\Gamma_{\omega}$ is the symplectic flux group.
Since $\Ker(\Flux_{\omega})$
has the fragmentation property
\cite{MR490874} (see also \cite{MR1445290}), %$\Ker(\Flux_{\omega})$ is a simple group. Thus
$\Ker(\Flux_{\omega})$ is normally generated by $\mathcal{S}_U$ and hence $\nu_U$ is well-defined on $\Ker(\Flux_{\omega})$.

\medskip

\item Let $\T_M$ be $\Diff_0(M,\vol)$ and $\Flux \colon \T_M \to H^{n-1}_c(M;\R)/\Gamma$ denote the volume flux homomorphism, where $\Gamma$ is the volume flux group.
Since $\Ker(\Flux)$ has the fragmentation property
(an unpublished result of W.~Thurston, see Banyaga's book \cite{MR1445290}),
$\nu_U$ is well-defined on $\Ker(\Flux)$.

\medskip

\item Let $\T_M$ be $\Homeo_0(M,\mu)$ and $\widetilde{\T}_M$ its universal covering, where $\T_M$ is endowed with the $C^0$-topology. In \cite{MR584082}, Fathi defined the homomorphism $\tilde{\theta} \colon \widetilde{\T}_M \to H_1(M;\R)$ and $\tilde{\theta}$ induces the mass flow homomorphism
$\theta \colon \T_M \to H_1(M;\R)/\Gamma$, where $\Gamma = \tilde{\theta}(\pi_1(\widetilde{\T}_M))$
(see also \cite{MR2377251}).
Since $\Ker(\theta)$ has the fragmentation property
\cite{MR584082},
$\nu_U$ is well-defined on $\Ker(\theta)$.
\end{itemize}

\end{example}

Henceforth $H^{\bullet}_c$ denotes the (de Rham) cohomology with compact support and $H^{\bullet}_c$ defines a covariant functor.

\begin{example}
    Let $M$, $U$, and $i \colon U \to M$ be as above.
    We can think of $\T_U$ as a subgroup of $\T_M$ by extending the diffeomorphism on $U$ to $M$.
\begin{itemize}

\item Let $\T_M=\Symp_0(M,\omega)$. If $i^{\ast} \colon H_c^1(U; \R) \to H_c^1(M; \R)$ is surjective,
we can see that $\nu_U$ is well-defined on $\T_M$ as follows. For $g \in \T_M$, there exists $h \in \T_U \subset \T_M$ such that $\Flux_{\omega}(g)=\Flux_{\omega}(h)$
since the flux homomorphism is surjective and
by the assumption of $i^{\ast}$.
Thus $g=(g h^{-1})  h$ is written as a product of the conjugation of the elements of
$\mathcal{S}_U$ since  $h \in \mathcal{S}_U $ and $gh^{-1} \in \Ker(\Flux_{\omega})$.

\begin{center}
  \begin{tikzpicture}[auto]
  \node (01) at (0, 1) {$\pi_1(\T_M)$}; %upper
  \node (11) at (3, 1) {$\widetilde{\T}_M$};
  \node (21) at (6, 1) {$\T_M$};
  \node (00) at (0, 0) {$\Gamma_{\omega}$}; %lower
  \node (10) at (3, 0) {$H^1_c(M;\R)$};
  \node (20) at (6, 0) {$H^1_c(M;\R)/\Gamma_{\omega}$};
  \draw[->] (00) to node {} (10); %upper
  \draw[->] (10) to node {} (20);
  \draw[->] (01) to node {} (11); %lower
  \draw[->] (11) to node {} (21);
  \draw[->] (01) to node {$\widetilde{\Flux}_{\omega}|_{\pi_1(\T_M)}$} (00); %vertical
  \draw[->] (11) to node {$\widetilde{\Flux}_{\omega}$} (10);
  \draw[->] (21) to node {$\Flux_{\omega}$} (20);
  \end{tikzpicture}
\end{center}

\medskip

\item Let $\T_M=\Diff_0(M,\vol)$. If $i^{\ast} \colon H_c^{n-1}(U; \R) \to H_c^{n-1}(M; \R)$ is surjective, we can see that $\nu_U$ is well-defined on $\T_M$ by the same argument.
\medskip

\item Let $\T_M=\Homeo_0(M,\mu)$. If $i_{\ast} \colon H_1(U; \R) \to H_1(M; \R)$ is surjective, we can see that $\nu_U$ is well-defined on $\T_M$ by the same argument.

 \end{itemize}

 \end{example}

Let $\pi_M$ denote the quotient group of the fundamental group of $M$ by its center.
%$\pi_1(M,z)/Z(\pi_1(M,z))$, where $z \in M$ is a base point and $Z(G)$ denotes the cent
The main theorem of this paper is the following.

\begin{thm} \label{thm:main}
Let $\T_M$ be $\Homeo_0(M,\mu)$, $\Diff_0(M, \vol)$, or $\Symp_0(M,\omega)$.
Assume that there exists an open subset $U$ of $M$ with finite volume such that $\nu_U$ is well-defined on $\T_M$.
If either
\begin{enumerate}
\item $ \pi_M$ surjects onto $F_2$ or
\item $\pi_M$ is an acylindrically hyperbolic group,
\end{enumerate}
then
  \[ \dim_{\R} EH^n_{\nu_U}(\T_M) \geq \dim_{\R}  \overline{EH}_b^n(F_2)\]
for every degree $n \geq 1$.
\end{thm}

Here, $\overline{EH}_b^n(G)$ denotes the reduced exact bounded cohomology of $G$ (see Section \ref{subsection:bounded_cohomology}).
%, and $\nu_U$ is the fragmentation norm with respect to $U$ (see Definition \ref{def:frag}).
If $M$ is of finite volume and $U=M$, this implies (a weak version of) the result of \cite[Theorem A and B]{brandenbursky2019bounded}.
%We remark that the condition (2) is satisfied if $\pi_M$ is acylindrically hyperbolic.
It is known that most 3-manifold groups are acylindrically hyperbolic \cite{MR3368093}.
On the other hand, if $M$ is 3-dimensional and $\pi_1(M)$ is finitely generated, there exists a 3-dimensional compact submanifold $C$ such that the inclusion $C \to M$ is a homotopy equivalence by the Scott core theorem \cite{MR326737}.
Thus we have many examples that satisfy the assumption of the above theorem.
For more details about acylindrically hyperbolic groups, see \cite{MR3430352}.

We also define a variant of the norm-controlled cohomology $H^n_{(d)}(G,\nu)$ for a non-negative integer $d$ such that $H^n_{(0)}(G,\nu)=H^n_{\nu}(G)$.
We will observe several basic properties of this cohomology, such as functoriality (Proposition \ref{prop:functor}). The cochain obtained by Burandenbursky--Marcinkowski's construction for infinite volume manifolds is regarded as an element of this cohomology too.

As a corollary of Theorem \ref{thm:main}, we obtain the following result.

\begin{cor} \label{cor:inf_dim}
  Suppose $M$ and $U$ satisfy the assumption in Theorem \ref{thm:main}. Then
   $EH_{(d)}^3(\mathcal{T}_M , \nu_U) $ is uncountably infinite-dimensional for $d=0,1,2$.
\end{cor}

\begin{figure}[h]
  \begin{center}
    \begin{tikzpicture}[auto]
    \node (01) at (0, 2) {quasimorphisms};
    \node (11) at (7, 2) {bounded cohomology};
    \node (00) at (0, 0) {relative quasimorphisms};
    \node (10) at (7, 0) {norm-controlled cohomology};
    \draw[->] (01) to node {generalization} (00);
    \draw[->] (01) to node {coboundary} (11);
    \draw[->] (11) to node {generalization} (10);
    \draw[->] (00) to node {coboundary} (10);
    \end{tikzpicture}
    \caption{Relation to other notions}
    \label{fig:notions}
  \end{center}
\end{figure}

Our cohomology is also motivated by relative quasimorphisms.
Let $G$ be a group and $\nu \colon G \to [0,\infty)$ a conjugation-invariant norm on $G$. %(see \S \ref{subsection:def_of_coh}).
A function $\phi \colon G \to \R$ is called a \emph{relative quasimorphism} with respect to a conjugation-invariant norm $\nu$ if there exists a constant $C \geq 0$ such that
\begin{equation} \label{eq:rel_qm}
|\phi (gh) -\phi(g)-\phi(h) | \leq C \cdot \min\{\nu(g),\nu(h)\}
\end{equation}
for all $g,h \in G$.
Relative quasimorphisms first appeared in the paper of Entov and Polterovich \cite{MR2208798}. They appear mainly in symplectic geometry, but they are also found in another context by the author \cite{MR3809596}.
For a quasimorphism $\phi$, its coboundary $\bar{\delta} \phi$ is considered as an element of the second bounded cohomology class. Similarly, a relative quasimorphism $\phi \colon G \to \R$ with respect to a conjugation-invariant norm $\nu$ can be seen as an element of $H^2_{(1)}(G,\nu)$ (see Figure \ref{fig:notions}).
We hope that our cohomology has applications to the study of relative quasimorphisms.

We note that the idea of the technique in Proposition \ref{prop:main_observation} comes from the construction of relative quasimorphisms by Kawasaki \cite{MR3523258}
and Brandenbursky and K\k{e}dra \cite{brandenbursky2018fragmentation}.
Moreover, our construction also can be seen as a generalization of \cite{brandenbursky2018fragmentation}.
We remark that both \cite{brandenbursky2018fragmentation} and \cite{brandenbursky2019bounded} can be seen as a generalization of the construction of Gambaudo--Ghys \cite{GG04} and Polterovich \cite{MR2276956} (Figure \ref{fig:construction}).

\begin{figure}[h]
  \begin{center}
    \begin{tikzpicture}[auto]
    \node[align=center] (01) at (0, 2) {Gambaudo--Ghys\cite{GG04}, \\ Polterovich\cite{MR2276956}};
    \node (11) at (7, 2) {Brandenbursky--Marcinkowski \cite{brandenbursky2019bounded}};
    \node (00) at (0, 0) {Brandenbursky--K\k{e}dra \cite{brandenbursky2018fragmentation}};
    \node (10) at (7, 0) {our construction};
    \draw[->] (01) to node {generalization} (00);
    \draw[->] (01) to node {generalization} (11);
    \draw[->] (11) to node {generalization} (10);
    \draw[->] (00) to node {generalization} (10);
    \end{tikzpicture}
    \caption{Relation to previous constructions}
    \label{fig:construction}
  \end{center}
\end{figure}

\section{Preliminary on group cohomology} \label{section:group_coh}

In this section, we review the definitions and terminology of (bounded) cohomology of groups. We refer to \cite{MR2527432,MR3726870} for more information on this topic.
Throughout this paper, we only consider the cohomology with real coefficients.

\subsection{Group cohomology}

%In this subsection, we define the inhomogeneous complex $\bar{C}^\bullet(-)$ and the homogeneous complex $C^\bullet(-)$, and recall the correspondence between them.

Let $G$ be a group. We consider the space of (inhomogeneous) $n$-cochains
\[\bar{C}^n (G)=\{\bar{c} \colon G^n  \to \mathbb{R}\},\]
and the coboundary map $\bar{\delta} : \bar{C}^{n-1}(G) \to \bar{C}^{n}(G)$ defined by
\[\bar{\delta} \bar{c} (g_1, \dots , g_n) = \bar{c}(g_2 ,\dots, g_n ) + \sum_{i=1}^{n-1}(-1)^i \bar{c}(g_1,\dots, g_ig_{i+1},\dots,g_n) +(-1)^n \bar{c}(g_1,\dots,g_{n-1}) \]
for $\bar{c} \in \bar{C}^n(G)$ and $g_1 ,\dots, g_n \in G$. The cohomology of this cochain complex is called the \emph{(group) cohomology} of $G$ and denoted by $H^n(G)$.

There is another definition of this cohomology.
A map $c \colon G^{n+1} \to \mathbb{R}$ is said to be \emph{homogeneous} if $c( h g_0 ,\dots, h g_n )=c(g_0,\dots, g_n)$ for every $g_0,\dots, g_n,h \in G$.
The space of (homogeneous) $n$-cochains is
\[ C^n (G)=\{c \colon G^{n+1} \to \mathbb{R} \mid c \; \text{is homogeneous} \}, \]
and the coboundary map $\delta \colon C^{n-1}(G) \to C^n (G)$ is defined by
\[ \delta c (g_0, \dots , g_n) = \sum_{i=0}^{n}(-1)^i c(g_0,\dots, \widehat{g_i},\dots,g_n) \]
for $c \in C^n(G)$ and $g_0 ,\dots, g_n \in G$, where $\widehat{g_i}$ means that we omit the entry $g_i$.
The cohomology of $(C^n (G), \delta)$ also defines $H^n (G)$.

The correspondence between an inhomogeneous cochain $\bar{c} \in \bar{C}^n(G)$
and homogenous one $c \in C^n(G)$ is the following:
\begin{align}
  \label{eq:inhomog to homog}
\bar{c}(g_1,g_2,\dots,g_n) =  c(1,g_1,g_1g_2,\dots,g_1g_2\dots g_n), \\
  \label{eq:homog to inhomog}
c(g_0,g_1,\dots, g_n) = \bar{c}(g_0^{-1}g_1, g_1^{-1}g_2,\dots, g_{n-1}^{-1}g_n).
\end{align}

We call a cochain $c \in C^n(G)$ \emph{alternating} if
\[c(g_{\sigma(0)}, \dots , g_{\sigma(n)})=\sgn(\sigma)c(g_0,\dots,g_n)\]
for any $g_0,\dots,g_n \in G$ and $\sigma \in \mathfrak{S}_{n+1}$, where $\sgn(\sigma) \in \{\pm 1\}$ is the sign of $\sigma$.
Let $C^n_{\alt}(G)$ denote the set of alternating $n$-cochains. Then $(C^n_{\alt}(G), \delta)$ is a subcomplex of $(C^n(G), \delta)$. It is known that the cohomology of $(C^n_{\alt}(G), \delta)$ also defines $H^n(G)$.

\subsection{Bounded cohomology} \label{subsection:bounded_cohomology}

We review the definition of bounded cohomology. We only mention the inhomogeneous case but the homogenous case is defined similarly.
If we consider the subcomplex
\[\bar{C}^n_b (G)=\{\bar{c} \colon G^n  \to \mathbb{R} \mid \bar{c} \; \text{is bounded} \}\]
of $\bar{C}^n$, the homology of the complex $(\bar{C}^n_b (G), \bar{\delta})$ is called the \emph{bounded cohomology} of $G$ and denoted by $H^n_b(G)$.
The natural inclusion $\bar{C}^n_b(G) \to \bar{C}^n(G)$ induces the homomorphism
$H^n_b(G) \to H^n(G)$ called the \emph{comparison map}. The kernel of the comparison map $H^n_b(G) \to H^n(G)$ is called the \emph{exact bounded cohomology} and denoted by $EH^n_b(G)$.

For a cochain $\bar{c} \in \bar{C}^n_b (G)$, we define the \emph{norm} $\| \bar{c} \|$ of $\bar{c} $ by
\[ \| \bar{c} \| = \sup_{g_1,\dots,g_n \in G} | \bar{c} (g_1,\dots,g_n)|. \]
%This norm induces a natural seminorm on $H_b^n(G)$, which is also denoted by $\| \cdot \|$, defined by 
This norm induces a natural seminorm $\| \cdot \|$ on $H_b^n(G)$ defined by 
\[ \| \alpha \| = \inf \| \bar{c} \| \]
for $\alpha \in H^n_b(G)$, where the infimum is taken over all cochains $\bar{c} \in \bar{C}^n_b (G)$ representing the cohomology class $\alpha$. 

Let $N^n(G)$ denote the norm zero subspace of $H^n_b(G)$, i.e.,
\[N^n(G) = \{ \alpha \in H^n_b(G) \mid  \|\alpha \| = 0 \}. \]
The \emph{reduced cohomology} $\overline{H}_b^n(G)$ is defined by the quotient $H^n_b(G)/ N^n(G)$.
The \emph{reduced exact cohomology} $\overline{EH}_b^n(G)$ is defined by $EH^n_b(G)/ EN(G)$, where $EN^n(G)= N^n(G) \cap EH^n_b(G)$.

We can consider the homogeneous complex $C^{\bullet}_b(G)$, alternating homogenous and inhomogeneous subcomplex $C^{\bullet}_{b,\alt}(G)$ and $\bar{C}^{\bullet}_{b,\alt}(G)$, and they also define the cohomology $H^{\bullet}_b(G)$ endowed with its natural seminorm.

\section{Norm-controlled cohomology}

\subsection{Definition of the cohomology} \label{subsection:def_of_coh}

We refer to a group equipped with a pseudo norm as a \emph{normed group}.
We define the norm-controlled cohomology $H^{\bullet}_{(d)}(G,\nu)$ for a normed group $(G,\nu)$ and a non-negative integer $d$.

\begin{definition}
    For a cochain $\bar{c} \in \bar{C}^n(G)$ and a function $\mu \colon G^n \to [0,\infty)$, we say that $\bar{c}$ is \emph{Lipschitz with respect to $\mu$} if there exist constants $C,D \geq 0$ such that for every $g_1,\dots,g_n \in G$
      \[ |\bar{c} (g_1,\dots,g_n )| \leq C \cdot \mu(g_1,\dots,g_n ) + D. \]
\end{definition}

\begin{definition}
  For a normed group $(G,\nu)$ and non-negative integers $n$ and $d$, we define $\bar{C}^{n}_{(d)}(G,\nu)$ as follows.
  \begin{itemize}
    \item   If $n>d$, we define $\bar{C}^{n}_{(d)}(G,\nu)$ as the set of Lipschitz cochains $\bar{c} \in \bar{C}^n(G)$ with respect to $\nu_{(n,d)}$,
      where $\nu_{(n,d)} : G^n \to [0,\infty)$ is defined by
      \begin{align*}
        & \nu_{(n,d)} (g_1,\dots,g_n) =\min_{\substack{I \subset \{1,\dots,n\} \\ \# I=n-d} } \left\{ \sum_{i \in I}\nu(g_i) \right\} \\
     = & \min_{1\leq i_i<\dots<i_d\leq n}  \left\{ \nu(g_1)+\dots + \widehat{\nu(g_{i_1})}+\dots +\widehat{\nu(g_{i_d})} + \dots +\nu(g_n) \right\}.
    \end{align*}
    \item If $d \geq n$, we define $\bar{C}^{n}_{(d)}(G,\nu)=\bar{C}^n_b(G)$.
  \end{itemize}
\end{definition}

Note that $\bar{c} \in \bar{C}^{n}_{(0)}(G,\nu)$ implies
  \[|\bar{c}(g_1,\dots,g_n)| \leq C \cdot \{ \nu(g_1) + \dots + \nu(g_n) \} + D\]
and $\bar{c} \in \bar{C}^{n}_{(n-1)}(G,\nu)$ implies
  \[|\bar{c}(g_1,\dots,g_n)| \leq C \cdot \min \{ \nu(g_1), \dots, \nu(g_n) \} + D.\]

\begin{lem}
  For any integer $d\geq 0$, $(\bar{C}^{n}_{(d)}(G,\nu), \bar{\delta})$ is a subcomplex of $(\bar{C}^n(G) , \bar{\delta})$.
\end{lem}
\begin{proof}
  It is sufficient to prove that $\bar{\delta}(\bar{C}^{n-1}_{(d)}(G,\nu))\subset \bar{C}^{n}_{(d)}(G,\nu)$ for the case $n-1>d$.

Let $g_1,\dots,g_n$ be elements in $G$. It is easy to see that
\begin{gather*}
  \nu_{(n-1,d)}(g_2 ,\dots, g_n ) \leq \nu_{(n,d)}(g_1\dots,g_n), \\
  \nu_{(n-1,d)}(g_1 ,\dots, g_{n-1} ) \leq \nu_{(n,d)}(g_1,\dots,g_n).
\end{gather*}
Since $\nu(g_ig_{i+1}) \leq \nu(g_i) + \nu(g_{i+1})$,
\[\nu_{(n-1,d)}(g_1,\dots, g_ig_{i+1},\dots,g_n) \leq \nu_{(n,d)}(g_1,\dots, g_i, g_{i+1},\dots,g_n) \]
for $i=1,2,\dots, n-1$.
  Therefore, for $\bar{c} \in \bar{C}^{n-1}_{(d)}(G,\nu)$,
\begin{align*}
  & \quad |\bar{\delta}\bar{c}(g_1,\dots,g_n)| \\
  &\leq |\bar{c}(g_2 ,\dots, g_n )| + \sum_{i=1}^{n-1} |\bar{c}(g_1,\dots, g_ig_{i+1},\dots,g_n)| +|\bar{c}(g_1,\dots,g_{n-1})| \\
  &\leq (n+1)\{ C \cdot \nu_{(n,d)}(g_1,\dots, g_n) + D \}. \qedhere
\end{align*}
\end{proof}

\begin{definition}
For %a group $G$ equipped with a norm $\nu$
a normed group $(G,\nu)$
and an integer $d\geq 0$, we define the \emph{norm-controlled cohomology} $H^{\bullet}_{(d)}(G,\nu)$ \emph{of level $d$} to be the cohomology of the cochain complex $(\bar{C}^{n}_{(d)}(G,\nu), \bar{\delta})$.
\end{definition}

Note that the complexes $\{ \bar{C}^{n}_{(d)}(G,\nu) \}_{n,d}$ can be seen as a filtered complex, i.e., $\bar{C}^{n}_{(d)}(G,\nu) \subset \bar{C}^{n}_{(d')}(G,\nu)$ if $d \geq d'$.

By the correspondence (\ref{eq:homog to inhomog}),
we can define the homogeneous norm-controlled cochain complex $C^n_{(d)}(G,\nu)$ as the set of cochains $c \in C^n(G)$ which satisfies the following:
there exist constants $C,D \geq 0$ such that
\[ |c(g_1,\dots,g_n)| \leq  C \cdot \min_{\substack{I \subset \{1,\dots,n\} \\ \# I=n-d} } \left\{ \sum_{i \in I}\nu(g_{i-1}^{-1} g_i) \right\} + D.\]
We can also define the inhomogeneous (resp. homogeneous) alternating subcomplex $\bar{C}^{\bullet}_{(d),\alt}(G,\nu)$ (resp. $C^{\bullet}_{(d),\alt}(G,\nu)$) and they also define the cohomology $H^{\bullet}_{(d)}(G,\nu)$.

\begin{example}

Let $\mathbb{Z}^n$ be the free abelian group of rank $n$.
For a positive integer $l\leq  n$, define a (pseudo) norm $\nu_l$ on $\mathbb{Z}^n$ by
\[\nu_l(m_1,\dots,m_l,\dots,m_n) = |m_1| +\dots+|m_l| \]
for $m_1,\dots,m_n \in \mathbb{Z}$. Now we compute $H^1_{\nu_{l}}(\mathbb{Z}^n)$.
Note that $H^1_{\nu}(G)=H^1_{(0)}(G,\nu)=\mathrm{Ker} (\bar\delta \colon \bar{C}^1_{(d)} \to \bar{C}^2_{(d)} ) $ is the set of Lipschitz homomorphisms with respect to $\nu$.

We define a homomorphism $\phi_i \colon \mathbb{Z}^n \to \mathbb{R}$ by
$\phi_i (m_1,\dots,m_l,\dots,m_n) = m_i$.
$\mathrm{Hom}(\mathbb{Z}^n,\mathbb{R}) \cong \mathbb{R}^n$ is generated by $\phi_1, \dots , \phi_n$.
It is easy to see that $\phi_i$ is Lipschitz with respect to $\nu_{l}$ for $i\leq l$ and not for $i>l$.
Thus, $H^1_{\nu_{l}}(\mathbb{Z}^n)$ is generated by $\phi_1, \dots , \phi_l$ and isomorphic to $\mathbb{R}^l$.
\end{example}

\begin{example}
Let $\phi$ be an unbounded relative quasimorphism (see \eqref{eq:rel_qm}) on a group $G$ with respect to a conjugation-invariant norm $\nu$. If $\phi$ is at infinite distance from homomorphisms,
%(i.e., $\phi$ is not represented by the sum of a homomorphism and a bounded function),
then $\bar{\delta}\phi$ defines a non-trivial cohomology class in $EH^2_{(1)}(G,\nu)$. Thus, for the following cases, $EH^2_{b}(G)$ is trivial but $EH^2_{(1)}(G,\nu)$ is non-trivial for a certain norm $\nu$.
\begin{itemize}
  \item $G$ is the identity component of the group of symplectomorphisms $\Symp_0^c(\R^{2n},\omega_0)$ of the standard symplectic space $(\R^{2n},\omega_0)$ with compact support \cite{MR3523258}.
  \item $G$ is the infinite braid group $B_{\infty}$ \cite{MR3809596}.
  \item $G$ is the Hamiltonian diffeomorphism group ${\rm Ham}(T^{\ast}\Sigma_g \times \R^{2n})$ of $T^{\ast}\Sigma_g \times \R^{2n}$, where $\Sigma_g$ is a closed surface of genus $g > 1$ \cite{brandenbursky2018fragmentation}.
\end{itemize}
\end{example}

\subsection{Functoriality}

We show that our cohomology is a functor for a certain category.
Recall that a normed group $(G,\nu)$ is a pair of a group $G$ and a (pseudo) norm $\nu$ on $G$.

\begin{definition}
  Let $(G,\nu_G)$ and $(H,\nu_H)$ be normed groups.
  A homomorphism $\phi \colon G \to H$ is said to be \emph{Lipschitz} if there exist $C,D \geq 0$ such that for all $g\in G$,
  \[\nu_H(\phi(g)) \leq C \cdot \nu_G(g)+D. \]
\end{definition}

\begin{definition}
  We define the \emph{category $\mathbf{NGrp}$ of normed groups} as follows.
\begin{itemize}
  \item The objects $Ob(\mathbf{NGrp})$ are normed groups.
  \item The morphisms $Mor(\mathbf{NGrp})$ are Lipschitz homomorphisms $\phi \colon (G,\nu_G) \to (H,\nu_H)$ between normed groups
  $(G,\nu_G)$ and $(H,\nu_H)$.
\end{itemize}
\end{definition}

The composition of morphisms is the composition of group homomorphisms, and hence the associativity holds.
For every $(G,\nu) \in Ob(\mathbf{NGrp})$, there exists the identity $\mathrm{id}_G:(G,\nu)\to (G,\nu)$ in $Mor(\mathbf{NGrp})$. Hence $\mathbf{NGrp}$ is a category.

Let $H^{n}_{(d)}$ denote the correspondence from a normed group $(G,\nu)$ to $H^{n}_{(d)}(G,\nu)$.

\begin{prop}\label{prop:functor}
  The correspondence $H^{n}_{(d)}$ is a contravariant functor from the category of normed groups $\mathbf{NGrp}$ to the category of real vector spaces $\mathbf{Vect}_{\mathbb{R}}$.
\end{prop}

If two norms on a same group are bi-Lipschitz, then they define the same norm-controlled cohomology.

\begin{cor} \label{cor:biLip}
  Let $G$ be a group with (pseudo) norms $\nu_1$ and $\nu_2$. If both $\id_G \colon (G,\nu_1) \to (G,\nu_2)$ and $\id_G \colon (G,\nu_2) \to (G,\nu_1)$ are Lipschitz, then $\id_G^{\ast} \colon H^n_{(d)}(G,\nu_2) \to H^n_{(d)}(G,\nu_1)$
  is an isomorphism.
\end{cor}
\begin{comment}

\begin{proof}
  Let $\phi$ denote $\id_G \colon (G,\nu_1) \to (G,\nu_2)$
  and $\psi$ denote $\id_G \colon (G,\nu_2) \to (G,\nu_1)$.
  Note that $\psi \circ \phi$ is the identity morphism $\id_{(G,\nu_1)}$ for
  $(G,\nu_1) \in Ob(\mathbf{NGrp})$.
  Thus, $\phi^{\ast} \circ \psi^{\ast} \colon H^n_{(d)}(G,\nu_1) \to H^n_{(d)}(G,\nu_1)$ is the identity. Similarly, $\psi^{\ast} \circ \phi^{\ast} \colon H^n_{(d)}(G,\nu_2) \to H^n_{(d)}(G,\nu_2)$ is also the identity.
  Therefore, $\id_G^{\ast}=\phi^{\ast} \colon H^n_{(d)}(G,\nu_2) \to H^n_{(d)}(G,\nu_1)$  is an isomorphism.
\end{proof}
\end{comment}

The proof of Proposition \ref{prop:functor} and Corollary \ref{cor:biLip} are elementary and will be omitted. %left to the reader.

\section{Cocycles on transformation groups}

\subsection{Brandenbursky--Marcinkowski construction}

We briefly review the construction of Brandenbursky and Marcinkowski \cite{brandenbursky2019bounded}.
Let $M$ be a finite volume complete Riemannian manifold and $\mu$ the measure on $M$ associated with the Riemannian structure. Fix a base point $z \in M$.
Let $\mathrm{Homeo}_0(M, \mu)$ denote the identity component of the group of homeomorphisms of $M$ with compact support that preserve the measure $\mu$.
Recall that $\pi_M$ denotes the quotient group $\pi_1(M,z)/Z(\pi_1(M,z))$, where $Z(G)$ denotes the center of a group $G$.

For a subgroup $\mathcal{T}_{M}$ of $\mathrm{Homeo}_0(M,\mu)$, they constructed a map $\Gamma_b^\bullet \colon H_b^\bullet( \pi_M ) \to H_b^\bullet( \mathcal{T}_{M} ) $ as follows.
Let $C$ denote the cut locus of $z$, i.e., 
the closure of the set of all points $x\in M$ that have two or more distinct shortest paths in $M$ from $x$ to $z$. 
For $x \in M$ and $g \in \mathcal{T}_{M}$ such that $x \not\in C$ and $g(x) \not\in C$,
we define $\gamma(g,x) \in \pi_1(M,z)$ by the concatenation of the minimizing geodesic between $z$ and $x$, the path defined by $\{ g_t(x)\}_{0\leq t\leq 1}$, where
$\{ g_t \}_{0\leq t\leq 1}$ is an isotopy of $g$ with $g_0=\id$ and $g_1=g$, and the minimizing geodesic between $g(x)$ and $z$. Then $\gamma(g,x)$ is defined uniquely up to center for any choice of isotopies. Thus it defines an element of $\pi_M$.

Since the measure $\mu(C)$ of the cut locus $C$ is zero and the map $\gamma(f,\cdot) \colon M \to \pi_M$ has essentially finite image for $f \in \T_M$, we can define the map $\Phi_{BM} \colon C^n( \pi_M ) \to C^n( \mathcal{T}_{M} ) $ by
\[\Phi_{BM}(c)(g_0,\dots,g_n) = \int_M c(\gamma(g_0,x),\dots,\gamma(g_n,x)) d\mu(x)\]
for $c \in C^n( \pi_M )$ and $g_0,\dots,g_n \in \mathcal{T}_{M}$.
The map $\Gamma^n \colon H^n( \pi_M ) \to H^n( \mathcal{T}_{M})$ is defined as the induced map from
$\Phi_{BM}$.
If $c \in C_b^n(\pi_M)$ is a bounded cochain, then $\Phi_{BM}(c)$ is also a bounded cochain since
\[ |\Phi_{BM}(c)(g_0,\dots,g_n)| \leq \| c \| \cdot \vol(M) < +\infty \]
for any $g_0,\cdots,g_n \in \pi_M$.
Hence $\Phi_{BM}$ induces the map $\Gamma_b^n \colon H_b^n( \pi_M ) \to H_b^n( \mathcal{T}_{M} ) $. We also obtain the map of the exact part $E\Gamma_b^n \colon EH_b^n( \pi_M ) \to EH_b^n( \mathcal{T}_{M})$.

\subsection{Main result}

We consider the above construction for the case that $M$ has infinite volume.
In this case, the map $\Phi_{BM}$ is well-defined on $C_{b,\alt}(\pi_M)$ since we consider homeomorphisms with compact support. On the other hand, the image $\Phi_{BM}(c)$ of a bounded cochain $c \in C_{b,\alt}(\pi_M)$ might not be a bounded cochain. We prove that, however, the image is a norm-controlled cochain with respect to a fragmentation norm (Proposition \ref{prop:main_observation}).

%Now we prove that Brandenbursky--Marcinkowski's construction yields norm-controlled cochains.

\begin{prop} \label{prop:main_observation}
For $c \in C_{b,\alt}^{n}( \pi_M )$, there exists $C \geq0$ such that
\[|\Phi_{BM} (c)(g_0,\dots,g_n)| \leq C \cdot \min_{0\leq i<j\leq n} \{ \nu_U(g_i^{-1}g_j) \}. \]
for all $g_0,\dots, g_n \in \T_M$. In particular, $\Phi_{BM}(c) \in C^n_{(n-1),\alt}(\T_M,\nu_U)$.
\end{prop}

\begin{proof}
  We fix $i$ and $j$ ($0\leq i < j \leq n$).
  Assume that $\nu_U(g_i^{-1}g_j) = m$.
  Then we can write $g_i^{-1}g_j = (f_1^{-1}h_1f_1) \dots (f_m^{-1} h_m f_m) $, where $h_k \in \mathcal{S}_U$ and $f_k \in \T_M$ for $k = 1,\dots,m$.
  Take an isotopy $\{ (g_i)_t \}_{t}$ of $g_i$
  and isotopies $\{ (h_1)_t \}_{t}, \dots, \{ (h_m)_t \}_{t}$ for $h_1 ,\dots, h_m$ such that $\supp((h_k)_t) \subset U$ for every $t\in[0,1]$ and $k=1,\dots,m$.
  We define $(g_j)_t = (g_i)_t (f_1^{-1}(h_1)_tf_1) \cdots (f_m^{-1} (h_m)_t f_m)$.
  Then $\{ (g_j)_t \}_t$ is an isotopy of $g_j$.
  Set
  \[U_{ij} = \bigcup_{0\leq t\leq1} \supp \left((g_i)_t^{-1} (g_j)_t \right)=\bigcup_{0\leq t\leq1} \supp \left( (f_1^{-1}(h_1)_tf_1) \cdots (f_m^{-1} (h_m)_t f_m) \right).\]
   Note that $U_{ij} \subset f_1(U) \cup \dots \cup f_m(U)$. If $x \not\in U_{ij}$,
  $(g_i)_t(x)=(g_j)_t(x)$ for every $t \in [0,1]$.
  Thus $\gamma(g_i,x)=\gamma(g_j,x) \in \pi_M$.
  Since $c$ is alternating, $c( \gamma(g_0,x),\dots,\gamma(g_n,x))=0$.
  Therefore,
  \[ |\Phi_{BM}(c)(g_0,\dots,g_n)| \leq \vol(U_{ij}) \cdot \|c \| \leq m \cdot \vol(U) \cdot  \|c \|. \]
  Since we can arbitrarily take $i$ and $j$, the inequality holds for $C=\vol(U) \cdot  \|c \|$.
\end{proof}

\begin{remark}
  For $d \leq n-1$, the map $\Phi_{BM} \colon C^n_{b,\alt}(\pi_M) \to C^n_{(d),\alt}(\T_M, \nu_U)$ is well-defined.
  However, if $d=n-1$, $\Phi_{BM}$ does not induce the map $H^n_{b}(\pi_M) \to H^n_{(n-1)}(\T_M, \nu_U)$ because the image of $\bar{B}^n_{(n-1)}(\pi_M)$ might not be in $\bar{B}^n_{(n-1)}(\T_M)$.
  On the other hand, if $d < n-1$,  $\Phi_{BM}$ induces $H^n_{b}(\pi_M) \to H^n_{(d)}(\T_M, \nu_U)$.
  Especially, if $d=0$, then $\Phi_{BM} \colon C^n_{b,\alt}(\pi_M) \to C^n_{\nu_U,\alt}(\T_M)$ induces $H^n_{b}(\pi_M) \to H^n_{\nu_U}(\T_M)$ for any $n \geq 2$.
\end{remark}

We prove the following key lemma which corresponds to \cite[Lemma 3.3]{brandenbursky2019bounded}.
Let $a$ and $b$ be generators of $F_2$.
\begin{lem}\label{lem:family of rep}
  Let $U \subset M$ be an open subset such that $\nu_U$ is well-defined on $\T_M$.
Assume that there exists an injection $i \colon F_2 \to \pi_M$.
Let $\alpha$ and $\beta$ be two loops in $M$ representing $i(a)$ and $i(b)$.
Suppose that $\alpha$ and $\beta$ are contained in $U$.
Then there exist a constant $\Lambda>0$ and a family of Lipschitz homomorphisms $\{\rho_{\epsilon} \colon (F_2, \nu_0) \to (\T_M, \nu_{U})\}_{\epsilon \in (0,1)}$  such that
\[ \lim_{\epsilon \to +0} \| \rho_{\epsilon}^{\ast} (E\Gamma_{\nu_U}(c)) - \Lambda i^{\ast}(c) \| = 0 \]
for every $c \in EH^n_{\nu_U}(\pi_M)$.
\end{lem}
Here, $\nu_0 \colon F_2 \to [0,\infty)$ is the trivial norm defined by
\[ \nu_0(w)=
\begin{cases}
  0 & (w = 1_{F_2}), \\
  1 & (w \neq 1_{F_2}).
\end{cases}
\]
The maps $i^{\ast}$ and $\rho_\epsilon^{\ast}$ represent the induced maps $i^{\ast} \colon EH^n_b(\pi_M) \to EH^n_b(F_2)$ and
$\rho_\epsilon^{\ast} \colon EH^n_{\nu_U}(\pi_M) \to EH^n_{\nu_0}(F_2)=EH^n_b(F_2)$.

\begin{proof}
  Our proof is similar to \cite[Lemma 3.3]{brandenbursky2019bounded}.   Let $N(\alpha)$ denote a tubular neighborhood of $\alpha$ in $U$ and take a diffeomorphism
  $n_{\alpha} \colon  N(\alpha) \to S^1 \times B^{n-1}(1) $. Here $B^{n-1}(r)$ denotes the $(n-1)$-ball in $\R^n$ with radius $r$. Let $A_{\epsilon}(\alpha)$ denote $n_{\alpha}^{-1}(S^1 \times B^{n-1}(1-\epsilon))$.
   We define an element $\rho_{\epsilon}(a) \in \T_M $ which ``rotates'' every point in $A_{\epsilon}$ one lap in the direction of $S^1$
  and fixes outside of $N(\alpha)$ (see \cite{brandenbursky2019bounded} for more details). Similarly, we define $N(\beta) \subset U$, $B_{\epsilon}$ and $\rho_{\epsilon}(b) \in \T_M $ 
  (see Figure \ref{fig:nbd}). Thus we obtain the representation $\rho_{\epsilon} \colon F_2 \to \T_M $.
  Since $\supp(\rho_{\epsilon}(w))$ is contained in $U$ for any $w \in F_2$,
  the map $\rho_{\epsilon} \colon (F_2, \nu_0) \to (\T_M , \nu_U) $ is a Lipschitz homomorphism.
By the functoriality of the correspondence $H_{(0)}^n$ (Proposition \ref{prop:functor}), the map $\rho_{\epsilon}^{\ast} \colon EH^n_{\nu_U}(\T_M) \to EH^n_b(\pi_M)$ is well-defined.

\begin{figure}[h]
  \centering
  \includegraphics[keepaspectratio, scale=0.5]
       {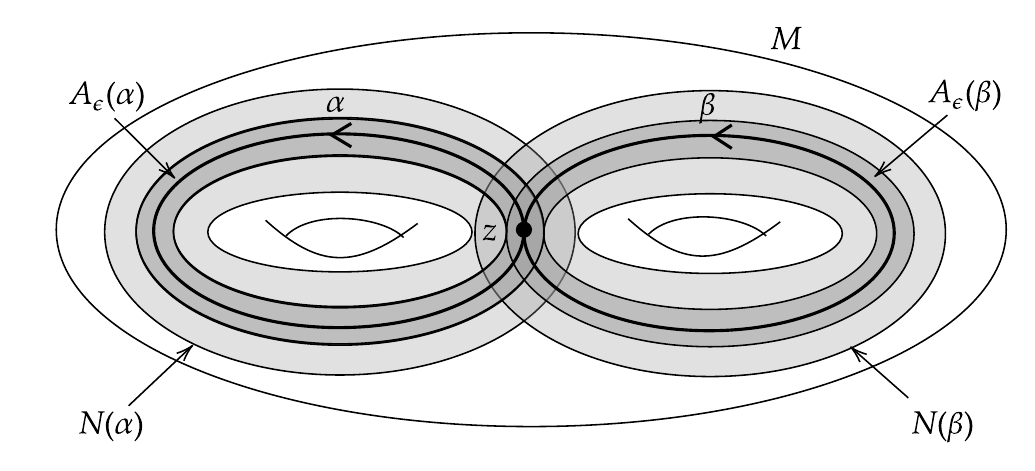}
  \caption{loops $\alpha$ and $\beta$ and their neighborhood}
  \label{fig:nbd}
 \end{figure}

For $w_0,\dots,w_n \in F_2$, we have
\[ \rho_{\epsilon}^{\ast} (E\Gamma_{\nu_U}(c))(w_0,\dots,w_n) = \int_M c (\gamma( \rho_{\epsilon}(w_0) , x),\dots,\gamma( \rho_{\epsilon}(w_n) , x) ) d\mu(x). \]
Let $B_{\epsilon}(\alpha)$ and $B_{\epsilon}(\beta)$ denote $N(\alpha)-A_{\epsilon}(\alpha)$ and $N(\beta)-A_{\epsilon}(\beta)$ respectively. We calculate this integral by decomposing $M$ into 5 parts;
$A_{\epsilon}:= A_{\epsilon}(\alpha) \cap A_{\epsilon}(\beta)$,  $A^a_{\epsilon}:= A(\alpha)-N(\beta)$,
$A^b_{\epsilon}:=A_{\epsilon}(\beta)-N(\alpha)$, $B_{\epsilon}:=B_{\epsilon}(\alpha)\cup B_{\epsilon}(\beta)$, and their exterior $M-(N(\alpha)\cup N(\beta))$ (see Figure \ref{fig:5parts}).

The exterior part is 0 and it turns out that $A^a_{\epsilon}$ and $A^b_{\epsilon}$ part are also 0.
The $A_{\epsilon}$ part is calculated to be $\mu(A_{\epsilon}) i^{\ast}(c) $ and the $B_{\epsilon}$ is bounded by $\mu(B_{\epsilon}) \| c \|$.
Hence the claim follows from $\mu(A_{\epsilon}) \xrightarrow{\epsilon \to +0} \mu(N(\alpha) \cap N(\beta)) >0 $ and $\mu(B_{\epsilon}) \xrightarrow{\epsilon \to +0} 0$.
\end{proof}

\begin{figure}[h]
 \centering
 \includegraphics[keepaspectratio, scale=0.5]
      {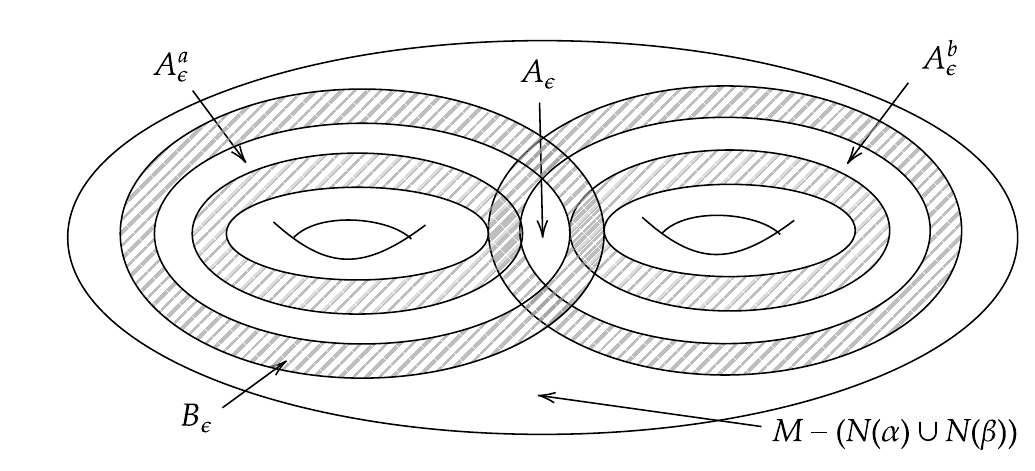}
 \caption{the decomposition of $M$ into 5 parts}
 \label{fig:5parts}
\end{figure}

We give proof of the main result. Our proof is similar to the proof of Theorem A and B in \cite{brandenbursky2019bounded}.

\begin{proof}[Proof of Theorem \ref{thm:main}]
  First, we prove case (1). Let $p\colon \pi_M \to F_2$ be a surjection.
  Assume that $\dim(M)\geq 3$. Then there exists an injection $i \colon F_2 \to \pi_M$ such that $p \circ i = \id_{F_2}$. If $\dim(M)=2$, we can find an injection $i \colon F_2 \to \pi_M$, and there exists a retraction $p\colon \pi_M \to F_2$ of $i$; we use this $p$ instead of the given $p$. If necessary, we retake $U$ to be containing $\alpha$ and $\beta$ in Lemma \ref{lem:family of rep}.

\begin{center}
  \begin{tikzpicture}[auto]
  \node (01) at (0, 2) {$EH^n_b(\pi_M)$};
  \node (11) at (3, 2) {$EH_{\nu_U}^n(\T_M)$};
  \node (00) at (0, 0) {$EH^n_b(F_2)$};
  \draw[->] (11) to node {$\rho_{\epsilon}^{\ast}$} (00);
  \draw[->] (01) to node {$E\Gamma_{\nu_U}$} (11);
  \draw[->, transform canvas={xshift=2pt}] (01) to node {$i^{\ast}$} (00);
  \draw[->, transform canvas={xshift=-2pt}] (00) to node  {$p^{\ast}$} (01);
  \end{tikzpicture}
\end{center}

Note that $EH_{\nu_U}^n(\T_M) \supset \Image(E\Gamma_{\nu_U}\circ p^{\ast}) \cong H_b^n(F_2)/ \Ker(E\Gamma_{\nu_U}\circ p^{\ast})$.
  For $d \in \Ker(E\Gamma_{\nu_U}\circ p^{\ast})$, set $c= p^{\ast}(d)\in EH^n_b(\pi_M)$. Since $i^{\ast} \circ p^{\ast} = \id$, $ i^{\ast}(c) = i^{\ast} \circ  p^{\ast}(d) = d$.

By Lemma \ref{lem:family of rep}, there exist $\Lambda>0$ and a family of representation $\{ \rho_{\epsilon} \}$ such that
\[ \lim_{\epsilon \to +0} \| \rho_{\epsilon}^{\ast} (E\Gamma_{\nu_U}(c)) - \Lambda i^{\ast}(c) \| = 0. \]
Since $E\Gamma_{\nu_U}(c) = E\Gamma_{\nu_U}\circ p^{\ast}(d) = 0$,
$\|i^{\ast}(c)\|= \|d \|=0$. Hence $\Ker(E\Gamma_{\nu_U}\circ p^{\ast}) \subset EN^n(F_2)$.
Therefore,
\[ \dim_{\R} \left( H_b^n(F_2) \Big/ \Ker(E\Gamma_{\nu_U}\circ i^{\ast}) \right) \geq
\dim_{\R} \left( EH_b^n(F_2) \Big/ EN^n(F_2) \right)=\dim_{\R}  \overline{EH}_b^n(F_2)\]
and we complete the proof for (1).

Next, we prove case (2). If $\dim(M)=2$, we can use the argument in the proof of (1). Thus we can assume that $\dim(M)\geq 3$.
Since $\pi_M$ is an acylindrically hyperbolic group,
then there exists a hyperbolic embedding $j\colon  F_2 \times K \to \pi_M$ \cite{MR3589159}.
We define $s \colon F_2 \to F_2 \times K$ by $r(x)=(x,\id)$ for $x \in F_2$ and $i \colon F_2 \to \pi_M$ by $i = j \circ s$.
Since we assumed that $\dim(M)\geq 3$, $i$ is injective. If necessary, we retake $U$ to be containing $\alpha$ and $\beta$ in Lemma \ref{lem:family of rep}.

The induced map $j^{\ast} \colon  EH^n_b(\pi_M) \to  EH^n_b(F_2 \times K) $ is surjective \cite{MR3431671}.
Since $s^{\ast} \colon EH^n_b(F_2 \times K) \to EH^n_b(F_2)$ induces an isomorphism, $i^{\ast} = j^{\ast} \circ s^{\ast}$ is also surjective.

\begin{center}
  \begin{tikzpicture}[auto]
  \node (01) at (0, 2) {$EH^n_b(\pi_M)$};
  \node (11) at (3, 2) {$EH_{\nu_U}^n(\T_M)$};
  \node (00) at (0, 0) {$EH^n_b(F_2)$};
  \draw[->] (11) to node {$\rho_{\epsilon}^{\ast}$} (00);
  \draw[->] (01) to node {$E\Gamma_{\nu_U}$} (11);
  \draw[->] (01) to node {$i^{\ast}$} (00);
  \end{tikzpicture}
\end{center}

Note that $EH_{\nu_U}^n(\T_M) \supset \Image(E\Gamma_{\nu_U}) \cong H_b^n(\pi_M)/ \Ker(E\Gamma_{\nu_U})$.
 Let $c \in EH^n_b(\pi_M) $.
If $ E\Gamma_{\nu_U}(c) =0$, then
$\| i^{\ast}(c) \| = 0$ by Lemma \ref{lem:family of rep}.
Thus $\Ker(E\Gamma_{\nu_U}) \subset \Ker(q \circ i^{\ast})$, where $q  \colon EH^n_b(F_2) \to \overline{EH}^n_b(F_2)$ is the quotient map. Therefore,
\[\dim_{\R} \left( H_b^n(\pi_M) \Big/ \Ker(E\Gamma_{\nu_U}) \right) \geq  \dim_{\R} \left( EH_b^n(F_2) \Big/ \Ker(q \circ i^{\ast}) \right) .\]
Since $q \circ i^{\ast}$ is surjective,
$EH_b^n(F_2) / \Ker(q \circ i^{\ast}) \cong \overline{EH}_b^n(F_2)$ and we complete the proof.
\end{proof}

\begin{proof}[Proof of Corollary \ref{cor:inf_dim}]
Since the dimension of $\overline{EH}_b^3(F_2)$ is uncountably infinite \cite{MR1455524}, by Theorem \ref{thm:main},
$EH^3_{\nu_U}(\T_M)=EH_{(0)}^3(\T_M , \nu_U) $ is also uncountably infinite-dimensional.
For $d=1,2$, There is the natural map $EH_{(d)}^3(\T_M , \nu_U) \to EH_{(0)}^3(\T_M , \nu_U)$ induced by the inclusion $C_{(d)}^3(\T_M , \nu_U) \to C_{(0)}^3(\T_M , \nu_U)$.
Since $\Phi_{BM}(c) \in C_{(d)}^3(\T_M , \nu_U)$ for $c \in C^3_b(\pi_M)$ by Proposition \ref{prop:main_observation}, this map surjects onto $\Image(E \Gamma_{\nu_U}) \subset EH^3_{\nu_U}(\T_M)$. We can see that the dimension of $\Image(E \Gamma_{\nu_U})$
 is uncountably infinite in the proof of Theorem \ref{thm:main}, thus $EH_{(d)}^3(\T_M , \nu_U)$ is also uncountably infinite-dimensional.
\end{proof}

\section*{Acknowledgment}
The author would like to thank Morimichi Kawasaki for suggesting this work and for many helpful comments.
He would also like to thank Michael Brandenbursky and Micha{\l} Marcinkowski for their comments on the webinar.
He would like to thank Ryuma Orita for reading the manuscript and his advice on improving this paper.
The author is supported by JSPS KAKENHI Grant Number JP20H00114 and JST-Mirai Program Grant Number JPMJMI22G1.


\begin{thebibliography}{0}
  
  \bibitem{MR490874}
  A.~Banyaga, \emph{Sur la structure du groupe des diff\'{e}omorphismes qui
    pr\'{e}servent une forme symplectique}, Comment. Math. Helv. \textbf{53}
    (1978), no.~2, 174--227.
  
  \bibitem{MR1445290}
  \bysame, \emph{The structure of classical diffeomorphism groups}, Mathematics
    and its Applications, vol. 400, Kluwer Academic Publishers Group, Dordrecht,
    1997.
  
  \bibitem{brandenbursky2018fragmentation}
  Michael Brandenbursky and Jarek K\k{e}dra, \emph{Fragmentation norm and
    relative quasimorphisms}, Proc. Amer. Math. Soc. \textbf{150} (2022), no.~10,
    4519--4531. \MR{4470192}
  
  \bibitem{brandenbursky2019bounded}
  Michael Brandenbursky and Micha{\l} Marcinkowski, \emph{Bounded cohomology of
    transformation groups}, Math. Ann. \textbf{382} (2022), no.~3-4, 1181--1197.
    \MR{4403221}
  
  \bibitem{MR2509711}
  D.~Burago, S.~Ivanov, and L.~Polterovich, \emph{Conjugation-invariant norms on
    groups of geometric origin}, Groups of diffeomorphisms, Adv. Stud. Pure
    Math., vol.~52, Math. Soc. Japan, Tokyo, 2008, pp.~221--250.
  
  \bibitem{MR2527432}
  D.~Calegari, \emph{scl}, MSJ Memoirs, vol.~20, Mathematical Society of Japan,
    Tokyo, 2009.
  
  \bibitem{MR3589159}
  F.~Dahmani, V.~Guirardel, and D.~Osin, \emph{Hyperbolically embedded subgroups
    and rotating families in groups acting on hyperbolic spaces}, Mem. Amer.
    Math. Soc. \textbf{245} (2017), no.~1156, v+152.
  
  \bibitem{MR2208798}
  M.~Entov and L.~Polterovich, \emph{Quasi-states and symplectic intersections},
    Comment. Math. Helv. \textbf{81} (2006), no.~1, 75--99.
  
  \bibitem{MR584082}
  A.~Fathi, \emph{Structure of the group of homeomorphisms preserving a good
    measure on a compact manifold}, Ann. Sci. \'{E}cole Norm. Sup. (4)
    \textbf{13} (1980), no.~1, 45--93.
  
  \bibitem{MR3726870}
  R.~Frigerio, \emph{Bounded cohomology of discrete groups}, Mathematical Surveys
    and Monographs, vol. 227, American Mathematical Society, Providence, RI,
    2017.
  
  \bibitem{MR3431671}
  R.~Frigerio, M.~B. Pozzetti, and A.~Sisto, \emph{Extending higher-dimensional
    quasi-cocycles}, J. Topol. \textbf{8} (2015), no.~4, 1123--1155.
  
  \bibitem{GG04}
  Jean-Marc Gambaudo and \'{E}tienne Ghys, \emph{Commutators and diffeomorphisms
    of surfaces}, Ergodic Theory Dynam. Systems \textbf{24} (2004), no.~5,
    1591--1617.
  
  \bibitem{MR686042}
  M.~Gromov, \emph{Volume and bounded cohomology}, Inst. Hautes \'{E}tudes Sci.
    Publ. Math. (1982), no.~56, 5--99 (1983).
  
  \bibitem{MR3523258}
  M.~Kawasaki, \emph{Relative quasimorphisms and stably unbounded norms on the
    group of symplectomorphisms of the {E}uclidean spaces}, J. Symplectic Geom.
    \textbf{14} (2016), no.~1, 297--304.
  
  \bibitem{MR3809596}
  M.~Kimura, \emph{Conjugation-invariant norms on the commutator subgroup of the
    infinite braid group}, J. Topol. Anal. \textbf{10} (2018), no.~2, 471--476.
  
  \bibitem{MR3368093}
  A.~Minasyan and D.~Osin, \emph{Acylindrical hyperbolicity of groups acting on
    trees}, Math. Ann. \textbf{362} (2015), no.~3-4, 1055--1105.
  
  \bibitem{MR2109110}
  C.~Ogle, \emph{Polynomially bounded cohomology and discrete groups}, J. Pure
    Appl. Algebra \textbf{195} (2005), no.~2, 173--209.
  
  \bibitem{MR2377251}
  Y.~Oh and S.~M\"{u}ller, \emph{The group of {H}amiltonian homeomorphisms and
    {$C^0$}-symplectic topology}, J. Symplectic Geom. \textbf{5} (2007), no.~2,
    167--219. \MR{2377251}
  
  \bibitem{MR3430352}
  D.~Osin, \emph{Acylindrically hyperbolic groups}, Trans. Amer. Math. Soc.
    \textbf{368} (2016), no.~2, 851--888.
  
  \bibitem{MR2276956}
  L.~Polterovich, \emph{Floer homology, dynamics and groups}, Morse theoretic
    methods in nonlinear analysis and in symplectic topology, NATO Sci. Ser. II
    Math. Phys. Chem., vol. 217, Springer, Dordrecht, 2006, pp.~417--438.
  
  \bibitem{MR326737}
  G.~P. Scott, \emph{Compact submanifolds of {$3$}-manifolds}, J. London Math.
    Soc. (2) \textbf{7} (1973), 246--250.
  
  \bibitem{MR1455524}
  T.~Soma, \emph{Bounded cohomology and topologically tame {K}leinian groups},
    Duke Math. J. \textbf{88} (1997), no.~2, 357--370.
  
  \end{thebibliography}
\end{document}